\documentclass[12pt,reqno]{amsart}

\usepackage[color]{showkeys}     % refs and labels
\definecolor{refkey}{gray}{.5}   % graylevel for refs
\definecolor{labelkey}{gray}{.5} % graylevel for labels

\usepackage{epsfig}
\usepackage{amsmath, amsthm, amsfonts, amssymb, mathrsfs}
\usepackage{graphicx}

\numberwithin{equation}{section}

\newcommand{\R}{{\mathbb R}}
\newcommand{\C}{{\mathbb C}}
\newcommand{\N}{{\mathbb N}}

\newcommand{\Z}{{\mathbb Z}}
\newcommand{\T}{{\mathbb T}}
\newcommand{\K}{{\mathbb K}}

\newtheorem{theo}{{\sc \bf Theorem}}[section]
\newtheorem{cor}[theo]{{\sc \bf Corollary}}
\newtheorem{lem}[theo]{{\sc \bf Lemma}}
\newtheorem{prop}[theo]{{\sc \bf Proposition}}

\newenvironment{rem}{\medskip\noindent{\bf Remark:\/} }{\medskip}
\newenvironment{defin}{\medskip\noindent{\bf Definition:\/} }{\medskip}

\begin{document}

\title{Unbounded Derivations in Algebras Associated with Monothetic Groups}

\author[Klimek]{Slawomir Klimek}
\address{Department of Mathematical Sciences,
Indiana University-Purdue University Indianapolis,
402 N. Blackford St., Indianapolis, IN 46202, U.S.A.}
\email{sklimek@math.iupui.edu}

\author[McBride]{Matt McBride}
\address{Department of Mathematics and Statistics,
Mississippi State University,
175 President's Cir., Mississippi State, MS 39762, U.S.A.}
\email{mmcbride@math.msstate.edu}

\thanks{}

\date{\today}

\begin{abstract}
Given an infinite, compact, monothetic group $G$ we study decompositions and structure of unbounded derivations in a crossed product C$^*$-algebra $C(G)\rtimes\Z$ obtained from a translation on $G$ by a generator of a dense cyclic subgroup. We also study derivations in a Toeplitz extension of the crossed product and the question whether unbounded derivations can be lifted from one algebra to the other.
\end{abstract}

\maketitle
\section{Introduction}

Derivations naturally arise in studying differentiable manifolds, in representation theory of Lie groups and in their noncommutative analogs. They also appear in mathematical aspects of quantum mechanics, in particular in quantum statistical physics. Additionally, derivations are important in analyzing amenability and other structures of operator algebras. Good overviews are in \cite{B} and also in \cite{S}.

In this paper we study classification and decompositions of unbounded derivations in C$^*$-algebras associated to an infinite, compact, monothetic group $G$, which, by definition, is a Hausdorff topological group with a dense cyclic subgroup. A group translation on $G$ by a generator of a cyclic subgroup is a minimal homeomorphism and one algebra associated with $G$ is the crossed product C$^*$-algebra $B:=C(G)\rtimes\Z$ determined by the translation. This algebra can be naturally represented in the $\ell^2$-Hilbert space of the full orbit. If we consider the analogous algebra on the forward orbit only, we obtain a Toeplitz extension $A$ of the algebra $B$. When the group is totally disconnected those algebras are precisely Bunce-Deddens and Bunce-Deddens-Toeplitz algebras considered in \cite{KMRSW2}.

The main objects of study in this paper are unbounded derivations $d: \mathcal A \rightarrow A$ which are defined on a subalgebra $\mathcal A$ of polynomials in generators of $A$.  Similarly, we study derivations $\delta: \mathcal B \rightarrow B$, where $\mathcal B$ is the image of $\mathcal A$ under the quotient map $A\to A/ \mathcal K = B$. The first of the main results of this paper is that any derivation in those algebras can be uniquely decomposed into a sum of a certain special derivation and an approximately inner derivation. The special derivations are not approximately inner, and can be explicitly described.

It turns out that any derivation $d: \mathcal A \rightarrow A$  preserves the ideal of compact operators $\mathcal K$ and consequently defines a factor derivation $[d]: \mathcal B \rightarrow B$ in $B$. It is an interesting and non-trivial problem to describe properties of the map $d\mapsto [d]$. For any C$^*$-algebra it is easy to see that bounded derivations preserve closed ideals and so they define derivations on quotients. It was proved in \cite{P} that for bounded derivations and separable C$^*$-algebras the above map is onto, i.e. derivations can be lifted from quotients. In non-separable cases this is not true in general. We prove here that lifting unbounded derivations from $B$ to $A$ is always possible when $G$ is totally disconnected, answering positively a conjecture in \cite{KMRSW2}. However we give a simple counterexample of a special derivation in the algebra $B$ for $G=\mathbb T^1$ that cannot be lifted to a derivation in the algebra $A$. Instead, we conjecture that for any compact, infinite, monothetic group approximately inner derivations in $B$ can be lifted to approximately inner derivation in $A$.

The paper is organized as follows.  In section 2 we review monothetic groups and discuss their properties.  We also describe a crossed product C$^*$-algebra that is associated to a monothetic group and that algebra Toeplitz extension,  as well as discuss a Toeplitz map from one algebra to another.  In section 3 we classify all unbounded derivations on polynomial domains in the C$^*$-algebras from section 2.  Finally, in section 4 we consider lifting derivations from a crossed product C$^*$-algebra to its Toeplitz extension. We prove that all derivations can be lifted for totally disconnected, compact, infinite, monothetic groups and provide an example that shows that not all derivations can be lifted in general.

\section{Monothetic Groups and Associated C$^*$-algebras}
\subsection{Monothetic Groups}
A topological (Hausdorff) group is called {\it monothetic} if it has a dense cyclic subgroup.  Andr{\'e} Weil observed, Theorem 19 of \cite{M}, that if $G$ is a locally compact monothetic group, then $G\cong\Z$ or $G$ is compact.  In this paper we only consider the case of compact $G$.  It follows immediately that $G$ is Abelian and separable.  We first describe the structure of such groups following \cite{HS}.  The key tool is the character (dual) group and Pontryagin duality, which translates properties of groups into properties of their duals.

Let $S^1$ be the unit circle:
\begin{equation*}
S^1 = \{z\in\C : |z|=1\},
\end{equation*} 
and let $\widehat{G}$ denote the dual group $G$, the group of continuous homomorphisms from $G$ to $S^1$ equipped with compact-open topology.  It is well known that if $G$ is compact then $\widehat{G}$ is discrete.  

We typically use additive notation for an abelian group, however we use multiplicative notation for the dual group.  Given a monothetic group $G$, let $x_1$ be a generator of a dense cyclic subgroup, and we set $x_n=nx_1$ for $n\in\Z$, so that $x_0:=0$ is the neutral element of $G$. Then we can identify the dual group $\widehat{G}$ of $G$ with a discrete subgroup of $S^1$ via the map given by:
\begin{equation*}
\widehat{G}\ni\chi\mapsto \chi(x_1)\in S^1.
\end{equation*}  
Conversely, using Pontryagin duality, if $H$ is a discrete subgroup of $S^1$, then $H$ is the dual group of a compact monothetic group, namely $\widehat{H}$, see \cite{HS}.

To better understand the structure of monothetic groups we look at the torsion subgroup of its dual group.  Given a monothetic $G$, the torsion subgroup of $\widehat{G}_{\textrm{tor}}$ of $\widehat{G}$ is given by:
\begin{equation*}
\widehat{G}_{\textrm{tor}} = \{\chi\in \widehat{G} : \chi^n =1\textrm{ for some }n\in\N\}.
\end{equation*}
There are two extreme cases: we say $\widehat{G}$ is of pure torsion if $\widehat{G} = \widehat{G}_{\textrm{tor}}$.  We also say $\widehat{G}$ is torsion free if $\widehat{G}_{\textrm{tor}} =\{0\}$. 
The following statements describe basic properties of monothetic groups.  We provide short or outlined proofs with references. A good, concise book on Pontryagin duality is \cite{M}.

First we look at the case of torsion free $\widehat{G}$.  

\begin{prop}
Let $G$ be a compact monothetic group.  $G$ is connected if and only if $\widehat{G}$ is torsion free.
\end{prop}

\begin{proof}
This is Corollary 4 of Theorem 30 of \cite{M}, which only requires $G$ to be compact, Abelian. 
\end{proof}

We have the following remarkable result proved in \cite{HS}. 

\begin{theo}\label{con_com_sep_mono}
Every connected compact separable Abelian topological group is monothetic.
\end{theo}

The $n$-dimensional torus, $\T^n = \R^n/\Z^n$ is an example of a compact, connected, separable, Abelian group and thus by Theorem \ref{con_com_sep_mono} is monothetic.  Consider an element 
$$x_1=(\theta_1,\ldots,\theta_n)$$ 
of $\T^n.$  Then the cyclic subgroup generated by $x_1$ is dense in $\T^n$  if and only if $\{1,\theta_1,\ldots,\theta_n\}$ are linearly independent over $\Z$, see for example \cite{KH}.

Next we consider the case when $\widehat{G}$ is of pure torsion.
\begin{prop}
Let $G$ be a compact monothetic group.  $G$ is totally disconnected if and only if $\widehat{G}$ is of pure torsion.
\end{prop}

\begin{proof}
This result follows for example from Corollary 1 of Theorem 30 of \cite{M}, since an element of the discrete group $\widehat{G}$ is compact (i.e. the smallest closed subgroup containing it is compact) if and only if it has finite order.
\end{proof}

Before we state the next structural result we need to introduce odometers. Further details on odometers can be found in \cite{D}.
The standard definition of an odometer (that inspired the name) uses a sequence of positive integers $b:=(b_m)_{m\in\N}$ such that $b_m\ge2$ for all $m$, called a multibase. The odometer is then identified (as a set) with the direct product:
$$G(b):=\prod_m \Z/b_m\Z,$$ 
but addition is defined with the carry over rule. Equipped with the product topology $G(b)$ becomes a compact, totally disconnected topological group. It is easy to see that the cyclic subgroup generated by
$x_1=(1,0,0,0,\ldots)$
is dense and so $G(b)$ is a monothetic group.

An alternative representation of the odometer $G(b)$ uses scales, and this is the description that is used in the proof of Theorem \ref{lift_theo}. Let $s = (s_m)_{m\in\N}$ be a sequence of positive integers such that $s_m$ divides $s_{m+1}$ and $s_m<s_{m+1}$.  There are natural homomorphisms between the consecutive finite cyclic groups $\Z/{s_{m+1}}\Z \to \Z/s_m\Z$, namely congruence modulo $s_m$. Thus the inverse limit:
$$G_s = \lim_{\underset{m\in \mathcal \N}{\longleftarrow}} \mathbb Z/ s_m\mathbb Z$$
is well defined as the subset of the countable product $\prod_m \Z/s_m\Z$ consisting of sequences $(y_1, y_2, y_3, \ldots)$ such that $y_{m+1}\equiv y_m\ (\textrm{mod } s_m)$. Addition in this representation is coordinate-wise, modulo $s_m$ in each coordinate $m$. $G_s$ becomes a topological group when endowed with the product topology over the discrete topologies in $\Z/s_m\Z$. Obviously, with our assumptions, this group is infinite because $s$ is unbounded. 

The relation between the two definitions of an odometer is as follows. Given a multibase $b=(b_m)_{m\in\N}$ define a scale $s= (s_m)_{m\in\N}$ by $s_1=b_1$, $s_2=b_1g_2$, $s_3=b_1b_2b_3$ and so on. Equivalently, we have:
$$s_1=b_1, \ b_n=\frac{s_n}{s_{n-1}} \textrm{ for } n>1.$$ 
Then the map
$$G(b)\ni(k_1,k_2,k_3,\ldots)\mapsto(k_1, k_1+k_2b_1, k_1+k_2b_1+k_3b_1b_2, \dots)\in G_s
$$
gives an isomorphism of the groups. In the scales representation of odometers the generator $x_1$ of a cyclic subgroup is given by $x_1=(1,1,1,1,\ldots).$

With the above definitions it is not transparent when two odometers are isomorphic, so we describe yet another way to define odometers that we used in \cite{KMRSW2}.  A {\it supernatural number} $N$ is defined as the formal product: 
\[N= \prod_{p-\textnormal{prime}} p^{\epsilon_p}, \;\;\; \epsilon_p \in\{0,1, \cdots, \infty\}.\]
If $\sum \epsilon_p < \infty$ then $N$ is said to be a finite supernatural number (a regular natural number), otherwise it is said to be infinite.  If 
$$N'= \prod_{p-\textnormal{prime}} p^{\epsilon_p'}$$ 
is another supernatural number, then their product is given by:
\[NN'= \prod_{p-\textnormal{prime}} p^{\epsilon_p + \epsilon_p'}.\]
A supernatural number $N$ is said to divide $M$ if $M=NN'$ for some supernatural number $N'$, or equivalently, if $\epsilon_p(N) \leq \epsilon_p(M)$ for every prime $p$.

Given a supernatural number $N$ let $\mathcal J_N$ be the set of finite divisors of $N$:
\[\mathcal J_N=\{j: \; j|N, j<\infty\}.\]
Then  $(\mathcal J_N, \leq)$ is a directed set where $j_1 \leq j_2$ if and only if $j_1 | j_2 |N$.
Consider the collection of  cyclic groups $\left\{\mathbb Z/ j\mathbb Z\right\}_{j \in \mathcal J_N}$ and the family of group homomorphisms
\[\begin{aligned}\pi_{ij}: \mathbb Z/ j\mathbb Z&\rightarrow \mathbb Z/ i\mathbb Z, \;\;\;\; j\geq i\\
\pi_{ij}(z)&=z\ (\textrm{mod } i)\end{aligned}\]
satisfying 
\[\pi_{ik} = \pi_{ij} \circ \pi_{jk} \textnormal{ for all } i \leq j \leq k.\]
Then the inverse limit of this system can be denoted as:
\[\mathbb Z/N\mathbb Z:=\lim_{\underset{j\in \mathcal J_N}{\longleftarrow}} \mathbb Z/ j\mathbb Z
=\left\{(z_j)\in\prod\limits_{j\in \mathcal J_N}\mathbb Z/ j\mathbb Z : \pi_{ij}(z_j)=z_i\right\}.\]
In particular, if $N$ is finite the above definition coincides with the usual meaning of the symbol $\mathbb Z/N\mathbb Z$, while if $N=p^\infty$ for a prime p, then the above limit is equal to $\mathbb Z_p$, the ring of $p$-adic integers, see for example \cite{Robert}. 

Given a scale $s= (s_m)_{m\in\N}$ we define the corresponding supernatural number $N$ to be the ``limit" of $s_m$:
\begin{equation}\label{N_limit}
N=\lim_{m\to\infty}s_m,
\end{equation}
in the sense that each prime exponent $\epsilon_p(N)$ of $N$ is defined to be the supremum of the prime exponents $\epsilon_p(s_m)$, $m\in\N$. It follows that $s_m$'s are divisors of $N$ and for every $j\in \mathcal J_N$ there is a natural number $m(j)$ such that $j|s_{m(j)}$. Consequently, a sequence $(z_j)\in \lim\limits_{\underset{j\in \mathcal J_N}{\longleftarrow}} \mathbb Z/ j\mathbb Z$ is completely determined by the subsequence $(z_{s_m})\in \lim\limits_{\underset{m\in \mathcal \N}{\longleftarrow}} \mathbb Z/ s_m\mathbb Z$, which gives an isomorphism $\Z/N\Z\cong G_s$. It turns out that odometers are classified by the supernatural number $N$, see \cite{D}. As before, 
$$x_1=(1,1,1,1,\ldots)\in\Z/N\Z$$ generates a dense cyclic subgroup.

In general, we have the following simple consequence of the Chinese Reminder Theorem:
if $N= \prod\limits_{\substack{p-\textnormal{prime} \\ {\epsilon_p \neq 0}}} p^{\epsilon_p}$, then 
$$ \mathbb Z/N\mathbb Z \cong \prod\limits_{\substack{p-\textnormal{prime} \\ {\epsilon_p \neq 0}}}  \mathbb Z/ {p^{\epsilon_p}}\mathbb Z.$$ 

Since the space $ \mathbb Z/N\mathbb Z $ is a compact, abelian topological group, it has a unique normalized Haar measure $\mu$. Also, if $N$ is an infinite supernatural number then $ \mathbb Z/N\mathbb Z$ is a Cantor set \cite{W}. 

We are now ready to state the next structural result about compact monothetic groups.

\begin{prop}\label{tot_disc_mono}
$G$ is a compact, totally disconnected, monothetic group if and only if it is an odometer. In particular, there exist a supernatural number $N$ such that
$$G\cong \Z/N\Z\cong \prod_i \Z/p_i^{\epsilon_i}\Z,$$ 
where $N = \prod_{i=1}^\infty p_i^{\epsilon_i}.$
\end{prop}

\begin{proof}
Let $G$ be a compact totally disconnected monothetic group.  In \cite{HS}, between Theorems $II'$ and $III$ on pages 256-257, the authors show that $G$ is isomorphic to a direct product of groups $G_{p_i}$ where $p_i$ runs over all primes and where $G_{p_i}$ isomorphic to the zero group, the cyclic group of order $p_i^{\epsilon_i}$ for some $\epsilon_i$ or the group of $p$-adic integers, the last case corresponds to $\epsilon_i=\infty$.
\end{proof}

In general, for arbitrary $\widehat{G}$ we have the following structure for compact monothetic groups.

\begin{prop}
Let $G$ be a compact monothetic group.  If $G_0\le G$ is the connected component of the neutral element $0$, then $G_0$ is a connected separable compact Abelian group and $G/G_0$ is a totally disconnected monothetic group.  Moreover, there are natural isomorphisms:
\begin{equation*}
\widehat{\left(G/G_0\right)}\cong \widehat{G}_{\textrm{tor}}\textrm{ and }\widehat{G_0}\cong \widehat{G}/\widehat{G}_{\textrm{tor}}
\end{equation*}
\end{prop}

\begin{proof}
This proposition is not formally stated but appears as a note in \cite{HS}, see also Corollary 3 of Theorem 30 of \cite{M}.  The first part follows from the previous propositions.  Recall that the annihilator $A(G_0)$ of $G_0$ is given by:
\begin{equation*}
A(G_0) = \{\chi\in \widehat{G} : \chi(g)=1,\textrm{ for all }g\in G_0\} .
\end{equation*}
By Pontryagin duality, Theorem 27 of \cite{M}, we have:
\begin{equation*}
A(G_0)\cong \widehat{\left(G/G_0\right)} .
\end{equation*}
Notice the right-hand side of the equation is Abelian, discrete and of pure torsion.  Thus given $\chi\in A(G_0)$, it defines a character class
\begin{equation*}
[\chi]: G/G_0\to S^1 .
\end{equation*}
Therefore, we have: 
\begin{equation*}
[\chi]\in\widehat{\left(G/G_0\right)} = \widehat{\left(G/G_0\right)}_{\textrm{tor}} ,
\end{equation*}
hence $\chi$ has finite order and thus $A(G_0)\le \widehat{G}_{\textrm{tor}}$.  Let $\chi\in\widehat{G}_{\textrm{tor}}$, then 
\begin{equation*}
\left.\chi\right|_{G_0}\in(\widehat{G_0})_{\textrm{tor}} = \{1\}
\end{equation*}
since $G_0$ is connected and thus $\chi\in A(G_0)$.  Therefore we have $A(G_0)=\widehat{G}_{\textrm{tor}}$ and hence
\begin{equation*}
\widehat{\left(G/G_0\right)}\cong \widehat{G}_{\textrm{tor}} .
\end{equation*}
The second isomorphism relation follows from Pontryagin duality: 
$$\widehat{G_0}\cong \widehat{G}/A(G_0)$$ 
and the proof is complete.
\end{proof}

\subsection{Minimal Systems}
By a {\it topological dynamical system} $(X,\varphi)$, we mean a topological space $X$ and a continuous map $\varphi:X\to X$, see \cite{KH}.  A topological dynamical system $(X,\varphi)$ is called {\it topologically transitive} if there exists a point $x\in X$ such that its orbit $\{\varphi^n(x)\}_{n\in\Z}$ is dense in $X$.  $(X,\varphi)$ is called {\it minimal} if every orbit is dense in $X$.  We say and write $\varphi$ is minimal for brevity.

Other equivalent characterization of minimal maps is as follows.  A set $A\subseteq X$ is called $\varphi$-invariant if $\varphi(A)\subseteq A$. Then, $\varphi$ is minimal if $X$ does not contain any non-empty, proper, closed $\varphi$-invariant subset. If in addition $X$ is assumed to be Hausdorff and compact, then a minimal map $\varphi$ must be surjective.   Moreover, if $(X,\varphi)$ is topologically transitive then there is no $\varphi$-invariant nonconstant continuous function on $X$.

Suppose that $G$ is a compact monothetic group with $x_1$ the generator of a dense cyclic subgroup. Then we define the map $\varphi:G\to G$ by the formula:
$$\varphi(x) = x+x_1.$$  
It follows that $(G,\varphi)$ is a minimal system.  Let us remark that for metrizable spaces a minimal, equicontinuous, dynamical systems coincide with translations by a generator of a dense cyclic subgroup of a compact monothetic groups, see Theorem 2.4.2 in \cite{K}. 

We now turn our attention to the algebras that are present in this paper.  Let $G$ be a compact infinite monothetic group, $C(G)$ the complex-valued continuous functions on $G$ and $\mu$ a normalized Haar measure on $G$.   Recall the notation for the elements of the cyclic subgroup generated by $x_1$:
\begin{equation*}
x_n = \varphi^n(0) = nx_1,
\end{equation*}
for $n\in\Z$.  The set $\{x_n\}_{n\in\Z}$ is the full orbit of $0$ under $\varphi$ and $\{x_n\}_{n\ge0}$ is the forward orbit.   As mentioned above, since $\varphi$ is a minimal homeomorphism, the forward orbit $\{x_n\}_{n\ge0}$ is dense in $G$.  

Consider the algebra of trigonometric polynomials on $G$:
\begin{equation*}
\mathcal{F} = \left\{\sum_nc_n\chi_n: \chi_n\in \widehat{G}, \textrm{ finite sums}\right\}.
\end{equation*} 
We state below two simple but useful properties of $\mathcal{F}$ that we will need later in the paper.  First we have the following observation.

\begin{prop}\label{property_F}
Let $f\in\mathcal{F}$.  Then $\int_Gf\,d\mu=0$ if and only if there is a trigonometric polynomial $g\in\mathcal{F}$ such that 
$$f = g\circ\varphi-g.$$
\end{prop}

\begin{proof}
If $f\in\mathcal{F}$ then $f$ has the following decomposition:
\begin{equation*}
f = \sum_j c_j\chi_j,
\end{equation*}
where $\chi_j$ are characters on $G$. Notice that we have
\begin{equation*}
\int_G\chi_j\,d\mu = \left\{
\begin{aligned}
&1 &&\textrm{ if }\chi_j =1 \\
&0 &&\textrm{else,}
\end{aligned}\right.
\end{equation*}
which means that $\int_Gf\,d\mu=0$ if and only if $\chi_j\neq1$ for all $j$.  Let $\chi$ be a nontrivial character, then the goal is to find a $g\in\mathcal{F}$ such that 
\begin{equation}\label{cocyc}
\chi(x) = g(x+x_1)-g(x).
\end{equation}
 Notice that for a nontrivial character we must have $\chi(x_1)\neq 1$. Otherwise, if $\chi(x_1)=1$, then $\chi(x_n)=\chi(nx_1)=1$ which in turn implies that $\chi=1$ on a dense set, and thus $\chi\equiv1$, which is a contradiction.  Therefore, we can choose 
\begin{equation*}
g(x) = \frac{\chi(x)}{\chi(x_1)-1},
\end{equation*}
which clearly satisfies \eqref{cocyc}.
Now that we can find a function $g(x)$ that solves \eqref{cocyc} for a nontrivial character, we just take finite linear combinations of such functions for the general case of a trigonometric polynomial, thus completing the proof.
\end{proof}

Next we describe another useful property of the space $\mathcal{F}$.  

\begin{prop}\label{property_G}
For any nonzero $n\in\Z$, there exists a trigonometric polynomial $f\in\mathcal{F}$ such that 
$$f-f\circ\varphi^n\neq0.$$
\end{prop}
\begin{proof}
The key property of the characters is that they separate points of $G$, see Theorem 14 of \cite{M}.  Therefore, if $n\ne 0$, we can pick $\chi$ such that:
\begin{equation*}
\chi(x_1)\neq\chi(x_{n+1}) = \chi(\varphi^n(x_1)).
\end{equation*}
As in the previous proposition, the general case is handled by linearity and the proof is complete.
\end{proof}

\subsection{C$^*$-algebras}
Let $G$ be an infinite, compact, monothetic group. We will describe now two types of C$^*$-algebras that can be naturally associated with such groups. They are defined as concrete C$^*$-algebras of operators in the following Hilbert spaces.  The first Hilbert space is the $\ell^2$ space of the full orbit: 
$$H=\ell^2(\{x_l\}_{l\in\Z}),$$ 
which is naturally isomorphic with $\ell^2(\Z)$.  Let $\{E_l\}_{l\in\Z}$ be the canonical basis in $H$.  The second Hilbert space is the $\ell^2$ space of the forward orbit:
$$H_+ = \ell^2(\{x_k\}_{k\in\Z_{\geq 0}})$$ 
which is naturally isomorphic with the Hilbert space $\ell^2(\Z_{\geq 0})$.  We also let $\{E^+_k\}_{k\in\Z_{\ge0}}$ be the canonical basis on $H_+$. 

The C$^*$-algebras associated to $G$ are defined using the following operators.  Let $V:H\to H$ be the shift operator on $H$:
$$VE_l = E_{l+1}.$$
 We also need the unilateral shift operator on $H_+$: 
 $$UE^+_k = E^+_{k+1}.$$ 
 Notice that $V$ is a unitary while $U$ is an isometry.  We have:
 $$[U^*,U]=P_0,$$
where $P_0$ is the orthogonal projection onto the one-dimensional subspace spanned by $E^+_0$.
 
 For a continuous function $f\in C(G)$ we define two operators $M_f:H\to H$ and $M_f^+:H_+\to H_+$ via formulas:
 $$M_fE_l = f(x_l)E_l\ \textrm{ and }\ M_f^+E^+_k = f(x_k)E^+_k.$$ 
They are diagonal multiplication operators on $H$ and $H_+$ respectively.  Due to the density of the orbit $\{x_k\}_{k\in\Z_{\geq 0}}$, we immediately obtain:

\begin{equation*}
\|M_f\|=\|M^+_f\|= \underset{l\in\Z}{\textrm{sup }}|f(x_l)| = \underset{k\in\Z_{\ge0}}{\textrm{sup }}|f(x_k)| =\underset{x\in G}{\textrm{sup }}|f(x)| = \|f\|_\infty.
\end{equation*}
The algebras of operators generated by $M_f$'s or by $M^+_f$'s are thus isomorphic to $C(G)$ so they carry all the information about the space $G$, while operators $U$ and $V$ reflect the dynamics $\varphi$ on $G$. The relation between those operators is:
$$VM_fV^{-1}=M_{f\circ\varphi}.$$
Similarly we have:
$$UM^+_f=M^+_{f\circ\varphi}U.$$
There is also another, less obvious relation between $U$ and $M^+_f$'s, namely:
\begin{equation}\label{M_with_P_zero}
M^+_fP_0=P_0M^+_f=f(x_0)P_0.
\end{equation}
We define the algebra $B$ to be the C$^*$-algebra generated by operators $V$ and $M_f$:
$$B=C^*\{V,M_f: f\in C(G)\}.$$ 
We claim that $B$ is isomorphic with the crossed product algebra:
$$B\cong C(G)\rtimes_\varphi \Z.$$ 
Indeed, observe that $\Z$ is amenable, the action of $\Z$ on $G$ given by $\varphi$ is a free action, and $\varphi$ is a minimal homoemorphism, thus the crossed product is simple and equal to the reduced crossed product, see \cite{F}. Clearly, the operators $V$ and $M_f$ define a representation of $C(G)\rtimes_\varphi\Z$, which must be isomorphic to it, by simplicity of the crossed product.  

The algebra $B$ has a natural dense $*$-subalgebra $\mathcal{B}$ of polynomials in $V$, $V^{-1}$, and the $M_\chi$'s, where $\chi$ is a character of $G$. Equivalently, we have:
\begin{equation*}
\mathcal{B} = \left\{\sum_nV^nM_{f_n}: f_n\in\mathcal{F},\textrm{ finite sums}\right\}.
\end{equation*}

Next we define the other algebra that is of the main interest in this paper, a Toeplitz extension of $B$. 
We define the algebra $A$ to be the C$^*$-algebra generated by operators $U$ and $M^+_f$:
$$A=C^*\{U,M^+_f: f\in C(G)\}.$$ 

To proceed further we need the following label operators on $H$ and $H_+$ respectively:  
$$\mathbb{L} E_l = lE_l\ \textrm{ and }\ \K E^+_k = kE^+_k.$$
The algebra $A$ has a natural dense $*$-subalgebra $\mathcal{A}$ of polynomials in $U$, $U^*$, $M^+_\chi$'s, where $\chi$ is a character of $G$, which can be equivalently described as follows, using Proposition 3.1 from \cite{KMRSW1} and also Proposition \ref{n_cov_set_eq} below:
\begin{equation*}
\mathcal{A} = \left\{\sum_{n\ge0}U^n(a_n^+(\K) + M_{f_n^+}^+) + \sum_{n\ge1}(a_n^-(\K)+M_{f_n^-}^+)(U^*)^n: f_n^\pm\in\mathcal{F},a_n^\pm(k)\in c_{00}(\Z_{\geq 0})\right\},
\end{equation*}
where the sums above are finite sums and $c_{00}(\Z_{\ge0})$ is the space of sequences that are eventually zero.
Notice that if $a\in\mathcal{A}$ and $x\in c_{00}(\Z_{\ge0})\subseteq H_+$, then $ax\in c_{00}(\Z_{\ge0})$, an observation that is often used below.

Next we establish the key relation between the two algebras $A$ and $B$.
Let $P_+:\mathcal{H}\to\mathcal{H_+}$ be the following map from $\mathcal{H}$ onto $\mathcal{H}_+$ given by
$$P_+E_k=\left\{
\begin{aligned}
&E_k^+ &&\textrm{ if } k\ge0\\
&0 &&\textrm{ if } k<0.
\end{aligned}\right.$$
We also need another map $s:\mathcal{H_+}\to\mathcal{H}$ given by:
$$sE^+_k  =E_k.$$
Define the map $T:B(\mathcal{H})\to B(\mathcal{H}_+)$, between the spaces of bounded operators on $\mathcal{H}$ and $\mathcal{H}_+$, in the following way: given $b\in B(\mathcal{H})$
\begin{equation*}
T(b) = P_+bs.
\end{equation*}
$T$ is known as a Toeplitz map.  It has the following properties.

\begin{prop}\label{alg_props_T}
Let $T$ be the Toeplitz map defined above. Then:
\begin{enumerate} 
\item $T(I_\mathcal{H}) = I_{\mathcal{H}_+}$. 
\item $T(bV^n) = T(b)U^n$ and $T(V^{-n}b) = (U^*)^nT(b)$ for $n\ge0$ and all $b\in B(\mathcal{H})$.
\item $T(bM_f) = T(b)M_f^+$ and $T(M_f\,b) = M_f^+T(b)$ for all $f\in C(X)$ and all $b\in B(\mathcal{H})$.
\end{enumerate}
Consequently, it follows that $T$ maps $B$ to $A$ and $\mathcal{B}$ to $\mathcal{A}$.
\end{prop}

\begin{proof}
For the first statement, if $h\in\mathcal{H}_+$ then we have the following calculation:
\begin{equation*}
T(I_\mathcal{H})h = (P_+sI_\mathcal{H})h = P_+h = h = I_{\mathcal{H}_+}h .
\end{equation*}
For the second statement we apply $T(bV^n)$ to the basis elements $E_k^+$ of $\mathcal H_+$.  We have
\begin{equation*}
T(bV^n)E_k^+ = (P_+bV^n)s(E_k^+) = (P_+b)E_{k+n} = (P_+bs)U^nE_k^+ = T(b)U^nE_k^+ .
\end{equation*}
A similar calculation shows the other equality $T(V^{-n}b) = (U^*)^nT(b)$.  Finally, for the last statement, we apply $T(bM_f)$ and $T(M_fb)$ to the basis elements to get:
\begin{equation*}
\begin{aligned}
T(bM_f)E_k^+ &= (P_+bM_f)s(E_k^+) = (P_+b)M_fE_k = (P_+b)f(x_k)E_k = (P_+bs)M_f^+E_k^+ \\
&= T(b)M_f^+E_k^+ .
\end{aligned}
\end{equation*}
This completes the proof.
\end{proof}

The next result describes the main relation between the two algebras $A$ and $B$.

\begin{prop}
The ideal of compact operators $\mathcal{K}$ in $B(\mathcal{H}_+)$ is an ideal in $A$.  Moreover, $B$ is the factor algebra:
\begin{equation*}
B\cong A/\mathcal{K}
\end{equation*}
and
\begin{equation}\label{T_iso}
B\ni b\mapsto T(b)+\mathcal{K}\in A/\mathcal{K}
\end{equation}
is an isomorphism.
\end{prop}

\begin{proof}
Notice first  that we have: 
$$P_0 =I - UU^*\in A.$$ 
It follows that the operators $P_{k,l} := U^kP_0(U^*)^l$ are also in $A$.  Thus, all finite rank operators with respect to the basis $\{E^+_k\}$ belong to $A$ as they are finite linear combinations of $P_{k,l}$.  Moreover, since all compact operators in $B(\mathcal{H}_+)$ are norm limits of these finite rank operators and $A$ is a C$^*$-algebra, it follows that $\mathcal{K}\subseteq A$.  It is clear that $\mathcal{K}$ is an ideal in $A$.  Verifying that the map given by equation (\ref{T_iso}) is an isomorphism, is analogous to the proof of Theorem 2.3 in \cite{KMRS}.
\end{proof}

It follows from the two previous propositions that we have the following identification.

\begin{cor}\label{structure_cor}
Under the isomorphism given by the equation (\ref{T_iso}), $\mathcal{B}$ is the factor algebra:
$$\mathcal{B}\cong [\mathcal{A}].$$
\end{cor}

For future reference we notice the following formulas: 
$$[U]=V, \ [M^+_f]=M_f, \ [\K]=\mathbb{L},$$
and also, for every $b\in B$:
$$[T(b)]=b.$$

Useful tools in classifying derivations on $A$ and $B$ are $1$-parameter groups of automorphisms of $A$ and $B$ respectively that are given by the following equations:
\begin{equation*}
\begin{aligned}
\rho_\theta^{\K}(a) &= e^{i\theta\K}ae^{-i\theta\K}&\textrm{ for }a\in A\\
\rho_\theta^{\mathbb{L}}(b) &= e^{i\theta\mathbb{L}}be^{-i\theta\mathbb{L}}&\textrm{ for }b\in B,
\end{aligned}
\end{equation*}
where $\theta\in\R/2\pi\Z$. We have the following formulas:
$$\rho_\theta^{\K}(U)=e^{i\theta}U,\ \rho_\theta^{\K}(a(\K))=a(\K),$$
and similarly for $\rho_\theta^{\mathbb{L}}$.
It immediately follows that $\rho_\theta^{\K}:\mathcal{A}\to\mathcal{A}$ and that $\rho_\theta^{\mathbb{L}}:\mathcal{B}\to\mathcal{B}$.  

The automorphisms define natural $\Z$-gradings on $A$ and $B$ given by the spectral subspaces:
\begin{equation*}
\begin{aligned}
A_n &=\{a\in A : \rho_\theta^{\K}(a) = e^{in\theta}a\}\\
B_n &=\{b\in B : \rho_\theta^{\mathbb{L}}(a) = e^{in\theta}b\}.
\end{aligned}
\end{equation*}
We call the elements of these sets the $n$-covariant elements of $A$ and $B$ respectively.  When $n=0$ we call those elements invariant.  

Let $c_0(\Z_{\ge0})$ be the space of sequences that converge to zero.  The $n$-covariant elements of $A$ and $B$ can described in detail.
\begin{prop}\label{n_cov_set_eq}
We have the following set equalities:
\begin{equation*}
A_n = \{a\in A : a=U^n(a_n(\K)+M_f^+), a(k)\in c_0(\Z_{\geq 0}), f\in C(G)\}
\end{equation*}
for $n\ge0$ and 
\begin{equation*}
A_n = \{a\in A : a=(a(\K)+M_f^+)(U^*)^{-n}, a(k)\in c_0(\Z_{\geq 0}), f\in C(G)\}
\end{equation*}
when $n<0$.
Similarly, we have:
\begin{equation*}
B_n = \{b\in B: b=V^nM_f, f\in C(G)\},
\end{equation*}
if $n\ge0$, and
\begin{equation*}
B_n = \{b\in B: b=M_fV^n, f\in C(G)\}
\end{equation*}
for $n<0$.
\end{prop}

\begin{proof}
Consider the invariant elements in $A$, that is $\rho_\theta^\K(a)=a$.  It follows from the definition of $\rho_\theta^\K$ that these elements are precisely the diagonal operators in $A$.  Moreover, we have the following unique decomposition, which is analogous to Proposition 2.4 in \cite{KMRSW2}:
\begin{equation*}
a=a(\K) + M_f^+
\end{equation*}
where $a(k)\in c_0(\Z_{\ge0})$ and $f\in C(G)$.
Next we consider the $n$-covariant elements for $n\neq0$.  Without loss of generality we only consider $n>0$.  Since we have:
\begin{equation*}
\begin{aligned}
\rho_\theta^{\K}(U^n) &= e^{in\theta}U^n\\
\rho_\theta^{\K}(a(\K) + M_f^+) &= a(\K) + M_f^+
\end{aligned}
\end{equation*}
for $a(k)\in c_0(\Z_{\geq 0})$ and $f\in C(G)$, one containment follows immediately.  On the other hand, if $a\in A_n$ then $a(U^*)^n$ is an invariant element and thus by the above has the form
\begin{equation*}
a(U^*)^n = a(\K) + M_f^+
\end{equation*} 
for some $a(k)\in c_0(\Z_{\geq 0})$ and $f\in C(G)$.  The other direction now follows.  The same argument also works for $B_n$, completing the proof.
\end{proof}

Similarly, we consider $n$-covariant elements from $\mathcal{A}$ and $\mathcal{B}$:
\begin{equation*}
\mathcal{A}_n=\{a\in\mathcal{A}: \rho_\theta^{\K}(a) = e^{in\theta}a\}\textrm{ and }\mathcal{B}_n=\{b\in\mathcal{B}: \rho_\theta^{\mathbb{L}}(b) = e^{in\theta}b\}.
\end{equation*}
As in Proposition \ref{n_cov_set_eq}, $a\in\mathcal{A}_n$ if and only if $a$ has the same element decomposition but with $a(k)\in c_{00}(\Z_{\ge0})$ and $f\in\mathcal{F}$.  Again, there is an analogous result for $b\in\mathcal{B}_n$.

\section{Classification of Derivations}
As in \cite{KMRSW2}, one of the main goals in this paper is to classify unbounded derivations in $A$ and $B$.  We begin with recalling the basic concepts.

Let $M$ be a Banach algebra and let $\mathcal{M}$ be a dense subalgebra of $M$.
A linear map $d:\mathcal{M}\to M$ is called a {\it derivation} if the Leibniz rule holds:
\begin{equation*}
d(ab) = ad(b) + d(a)b
\end{equation*}
for all $a,b\in\mathcal{M}$. We say a derivation $d:\mathcal{M}\to M$ is {\it inner} if there is an element $x\in M$ such that
\begin{equation*}
d(a) = [x,a]
\end{equation*}
for $a\in\mathcal{M}$.  We say a derivation $d:\mathcal{M}\to M$ is {\it approximately inner} if there are $x_n\in M$  such that 
\begin{equation*}
d(a) = \lim_{n\to\infty}[x_n,a]
\end{equation*} 
for $a\in\mathcal{M}$.  

Given $n \in \mathbb Z$, a derivation $d:\mathcal{A}\to A$ is said to be a {\it $n$-covariant derivation} if the relation 
 $$(\rho^{\mathbb K}_{\theta})^{-1}d(\rho^{\mathbb K}_{\theta}(a))= e^{-in\theta} d(a)$$ holds.  We have a similar definition for a derivation $\delta:\mathcal{B}\to B$.  Like above, when $n=0$ we say the derivation is invariant.

\subsection{Derivations in $A$}

We first classify all invariant derivations $d:\mathcal{A}\to A$.  An example of an invariant derivation is given by
\begin{equation*}
d_{\K}(a) = [\K,a]
\end{equation*}
where $a\in\mathcal{A}$.  This derivation is well defined because $\mathcal{A}$ is the space of polynomials in $U$, $U^*$, and $M_f^+$ and we have $[\K,U]=U$, $[\K,U^*]= -U^*$ and $[\K,M_f^+]=0$.

\begin{lem}\label{der_alpha}
For any $\{\alpha(k)\}\in c_0(\Z_{\ge0})$ there exists a unique derivation $d_{\alpha}:\mathcal{A}\to A$ such that
\begin{equation*}
d_{\alpha}(U) = U\alpha(\K),\quad d_{\alpha}(U^*)=-\alpha(\K)U^*,\quad d_{\alpha}(a(\K)) = 0
\end{equation*}
for every $a(\K)\in\mathcal{A}_0$.  Moreover this derivation is an approximately inner invariant derivation.
\end{lem}

\begin{proof}
Define a sequence $\{\alpha^N(k)\}$ as follows:
\begin{equation*}
\alpha_0^N(k) = \left\{
\begin{aligned}
&\alpha_0(k) &&\textrm{ for }k<N \\
&0 &&\textrm{ for }k\ge N.
\end{aligned}\right.
\end{equation*}
Then $\alpha^N(k)\in c_{00}(\Z_{\geq 0})$ and $\alpha^N(\K)$ converges to $\alpha(\K)$ as $N\to\infty$.  Next, define a sequence $\{\beta^N(k)\}$ by
\begin{equation*}
\beta^N(k) = \sum_{j=0}^{k-1}\alpha^N(j),
\end{equation*}
so that $\beta^N(k)$ is eventually constant.  Thus, $d_{\alpha^N}:\mathcal{A}\to A$ defined by
$$d_{\alpha^N}(a)=[\beta^N(\K),a]$$ 
is an invariant inner derivation.  We have
\begin{equation*}
\lim_{N\to\infty}d_{\alpha^N}(U) = U\alpha(\K)\,,\quad\lim_{N\to\infty}d_{\alpha^N}(U^*) = -\alpha(\K)U^*\textrm{ and }d_{\alpha^N}(a(\K))=0
\end{equation*}
for all $a(\K)\in\mathcal{A}_0$.  Thus, by the Leibniz rule, the limit 
\begin{equation*}
\lim_{N\to\infty} d_{\alpha^N}(a)=d_\alpha(a)
\end{equation*}
exists for all $a\in\mathcal{A}$.  Thus, this limit is a derivation from $\mathcal{A}$ to $A$.  It follows that $d_\alpha$ is approximately inner and invariant.
\end{proof}

\begin{lem}\label{der_f_zero}
For any $f\in C(G)$ such that
\begin{equation*}
\int_G f\,d\mu =0,
\end{equation*}
there exists a unique derivation $d_f:\mathcal{A}\to A$ such that 
\begin{equation*}
d_f(U) = UM_f^+,\quad d_f(U^*) = -M_f^+U^*,\quad d_f(a(\K)) = 0
\end{equation*}
where $a(\K)\in \mathcal{A}_0$.  Moreover $d_f$ is an approximately inner invariant derivation.
\end{lem}

\begin{proof}
By the density of $\mathcal{F}$ we can pick a sequence $\{f^N\}\subseteq\mathcal{F}$ such that $f^N$ converges to $f$ and
\begin{equation*}
\int_Gf^N\,d\mu=0.
\end{equation*}
By Proposition \ref{property_F}, there exists a sequence $\{g^N\}\subseteq\mathcal{F}$ such that 
$$f^N(x)=g^N(\varphi(x))-g^N(x).$$ 
We define
\begin{equation*}
d_{f^N}(a) = [M_{g^N}^+,a],
\end{equation*}
and notice that $d_{f^N}:\mathcal{A}\to A$ is an inner invariant derivation.  By direct calculation we have
\begin{equation*}
\lim_{N\to\infty} d_{f^N}(U) = UM_f^+\,,\quad\lim_{N\to\infty}d_{f^N}(U^*) = -M_f^+U^*,\textrm{ and }d_{f^N}(a(\K))=0,
\end{equation*}
for every $a(\K)\in\mathcal{A}_0$.  Thus, by the Leibniz rule, the limit 
\begin{equation*}
\lim_{N\to\infty} d_{f^N}(a)=d_f(a)
\end{equation*}
exists for all $a\in\mathcal{A}$ and is a derivation from $\mathcal{A}$ to $A$.  It follows that $d_f$ is approximately inner and invariant.
\end{proof}

\begin{prop}\label{invariant_der_A}
Given any invariant derivation $d:\mathcal{A}\to A$ there exists a number $c_0\in \C$ such that $d$ is of the unique form
\begin{equation*}
d(a) = c_0d_{\K}(a) + \tilde{d}(a)
\end{equation*}
where $\tilde{d}:\mathcal{A}\to A$ is approximately inner.
\end{prop}

\begin{proof}
Let $a(\K)\in\mathcal{A}_0$ be a diagonal operator such that $a(k)\in c_{00}(\Z_{\geq 0})$. Then, by invariance of $d$, we have $d(a(\K))\in A_0$.  Notice that since $A_0$ is precisely the algebra of diagonal operators in $A$ it is therefore a commutative algebra.  Let $P^2=P$ be a projection in $\mathcal{A}_0$.  Applying $d$ to both sides of the equation and using Leibniz's rule we have
\begin{equation*}
2Pd(P)=d(P^2)=d(P)
\end{equation*}
which implies that $(1-2P)d(P)=0$ and hence $d(P)=0$.  Since $a(\K)$ is a finite sum of projections in $\mathcal{A}_0$, it follows that $d(a(\K))=0$.

Let $P_k$ be the one-dimensional orthogonal projection onto the span of $E_k$.  Then $P_k\in\mathcal{A}_0$ and thus $d(P_k)=0$.  We have the following formula:
$$M_f^+P_k=f(x_k)P_k.$$ 
Therefore, applying $d$ to both sides, we obtain:
\begin{equation*}
d(M_f^+)P_k + M_f^+d(P_k) =f(x_k)d(P_k).
\end{equation*}
It follows that $d(M_f^+)P_k=0$ for all $k\in\Z_{\geq 0}$ and so, $d(M_f^+)=0$. It follows from Proposition \ref{n_cov_set_eq} that $d(a)=0$ for all $a\in\mathcal A_0$.

Notice that, by the invariance property of $d$, we have
\begin{equation}\label{d_inv_decomp}
\begin{aligned}
d(U) &= U(\alpha_0(\K) + M^+_{f_0}) \\
d(U^*) &=-(\alpha_0(\K) +M^+_{f_0})U^*
\end{aligned}
\end{equation}
for some $\alpha_0(k)\in c_0(\Z_{\geq 0})$ and $f_0\in C(G)$. 

Let $c_0$ be the following integral
\begin{equation*}
c_0=\int_G f_0\,d\mu.
\end{equation*}
and set $\tilde{f} = f_0 - c_0$ so that $\int_G \tilde{f}\,d\mu=0$.  By Lemmas \ref{der_alpha} and \ref{der_f_zero} and equation (\ref{d_inv_decomp}), we have the decomposition
\begin{equation*}
d = c_0d_\K + d_{\alpha_0} + d_{\tilde{f}}.
\end{equation*}
Picking $\tilde{d}=d_{\alpha_0} + d_{\tilde{f}}$ completes the proof of existence of the decomposition.

Finally, to verify uniqueness of the decomposition, we only need to check that that $d_\K$ is not approximately inner. If $d_\K$ is approximately inner then we can arrange that it can be approximated by inner invariant derivations of the form  $d_j(a) = [\beta_j(\K), a]$ with $\beta_j(k)\in c(\Z_{\geq 0})$. Since $\{\beta_j(k+1)-\beta_j(k)\}\in c_0(\Z_{\geq 0})$ we would also get $\{(k+1)-k\}\in c_0(\Z)$, which is a contradiction. Full details of an analogous result are given in Theorem 3.10 of \cite{KMRSW2}.
\end{proof}

Next we classify $n$-covariant derivations in $A$.

\begin{rem}
Let $n\neq0$ and by Proposition \ref{property_G} choose $f\in\mathcal{F}$ such that $f-f\circ\varphi^n\neq0$.  Since $f-f\circ\varphi^n$ is continuous on a compact set, the minimum is achieved and is not equal to zero.  Therefore we have
\begin{equation*}
\underset{x\in G}{\textrm{inf }}|f(x)-f(\varphi^n(x))|\neq0
\end{equation*}
and hence $M_f^+-M_{f\circ\varphi^n}^+$ is an invertible operator.  
\end{rem}
This remark is crucial for the proof of the next proposition.

\begin{prop}\label{covariant_der_A}
Let $d:\mathcal{A}\to A$ be an $n$-covariant derivation where $n\neq0$.  There exists an $\beta(\K)\in A_0$ such that
\begin{equation*}
\begin{aligned}
d(a)&=[U^n\beta(\K),a] &&\textrm{ for }n>0 \\
d(a)&=[\beta(\K)(U^*)^{-n},a] &&\textrm{ for }n<0,
\end{aligned}
\end{equation*}
and hence $d$ is an inner derivation.
\end{prop}

\begin{proof}
We only discuss the case of $n>0$ as the case of $n<0$ is completely analogous.  By definition of $n$-covariance there exists an $\alpha(\K)\in A_0$ such that
\begin{equation*}
d(U)=U^{n+1}\alpha(\K+I)\textrm{ and }d(U^*) = -U^{n-1}\alpha(\K).
\end{equation*}
We define a ``twisted'' derivation $\tilde{d}:\mathcal{A}_0\to A_0$ by $d(a(\K)) = U^n\tilde{d}(a(\K))$ for $a(\K)\in\mathcal{A}_0$.  A direct computation yields
\begin{equation*}
\begin{aligned}
\tilde{d}(a(\K)b(\K)) &= \tilde{d}(a(\K))b(\K) + a(\K+nI)\tilde{d}(b(\K)) \\
\tilde{d}(b(\K)a(\K)) &= \tilde{d}(b(\K))a(\K) + b(\K+nI)\tilde{d}(a(\K))
\end{aligned}
\end{equation*}
for $a(\K),b(\K)\in \mathcal{A}_0$.  Since $\mathcal{A}_0$ and $A_0$ are commutative algebras we get
\begin{equation*}
\tilde{d}(a(\K))\left(b(\K)-b(\K+nI)\right) = \tilde{d}(b(\K))\left(a(\K)-a(\K+nI)\right).
\end{equation*}
Similarly to the proof in Theorem 3.4 in \cite{KMRSW2}, there must exist a $\beta(\K)$ such that
\begin{equation*}
\tilde{d}(a(\K))=\beta(\K)\left(a(\K)-a(\K+nI)\right).
\end{equation*}
Consequently, we have the following formula
\begin{equation*}
d(a(\K)) = U^n\beta(\K)(a(\K) - a(\K+nI)).
\end{equation*}

Next we apply $d$ to the commutation relation $U^*a(\K) = a(\K+I)U^*$ for a diagonal operator $a(\K)\in\mathcal{A}_0$, and obtain:
\begin{equation*}
\begin{aligned}
&U^{n-1}(-\alpha(\K)a(\K) + \beta(\K)a(\K) - \beta(\K)a(\K+nI)) = \\
&\quad=U^{n-1}(\beta(\K-I)a(\K)-a(\K+nI)\beta(\K-I)-a(\K+nI)\alpha(\K)),
\end{aligned}
\end{equation*}
where we define $\beta(-1):=0$.  Rearranging these terms gives:
\begin{equation*}
\alpha(\K)(a(\K+nI)-a(\K)) = (\beta(\K)-\beta(\K-I))(a(\K+nI)-a(\K))
\end{equation*}
for all $a(\K)\in \mathcal{A}_0$.  It therefore follows that $\beta(\K)-\beta(\K-I) = \alpha(\K)$.  Thus $\beta(k)$ is uniquely determined by
\begin{equation*}
\beta(k) = \sum_{j=0}^k\alpha(j),
\end{equation*}
and it follows that
\begin{equation*}
d(a) = [U^n\beta(\K),a]
\end{equation*}
for any $a\in\mathcal{A}$, since both sides of the above equation are derivations, and they agree on the generators of the polynomial algebra $\mathcal{A}$. By the remark preceding the statement of the proposition, if $f\in\mathcal{F}$ is such that $M_f^+-M_{f\circ\varphi^n}^+$ is invertible, we can apply $d$ to $M_f^+$ to get
\begin{equation*}
A_0\ni d(M_f^+) = U^n\beta(\K)(M_f^+ - M_{f\circ\varphi^n}^+).
\end{equation*}
Therefore, it follows that we must have $\beta(\K)\in A_0$, and the proof is complete.
\end{proof}

To classify all derivations $d:\mathcal{A}\to A$ we need to define the Fourier coefficients of $d$ following the ideas of \cite{BEJ}.

 \begin{defin}\label{Fou_comp}
If $d$ is a derivation in $A$, the {\it $n$-th Fourier component} of $d$ is defined as: 
$$d_n(a)= \frac 1{2\pi} \int_0^{2\pi} e^{in\theta} (\rho^{\mathbb K}_{\theta})^{-1}d\rho^{\mathbb K}_{\theta}(a)\; d\theta.$$
 \end{defin}
A direct calculation shows that if $d:\mathcal{A}\to A$ is a derivation then $d_n:\mathcal{A}\to A$ is an $n$-covariant derivation.  

We have the following key Ces\`aro mean convergence result for Fourier components of $d$, which is more generally valid for unbounded derivations in any Banach algebra with the continuous circle action preserving the domain of the derivation: if $d$ is a derivation in $A$ then
\begin{equation}\label{Ces_eq}
d(a)=\lim_{M \rightarrow \infty } \frac 1{M+1} \sum_{j=0}^M \left(\sum_{n=-j}^j d_n(a)\right),
\end{equation}
for every $a\in \mathcal{A}$, see Lemma 4.2 in \cite{KMRSW2} for more details.

The following theorem classifies all derivations $d:\mathcal{A}\to A$.

\begin{theo}\label{der_decomp_A}
Let $d:\mathcal{A}\to A$ be any derivation.  Then there exists $c_0\in\C$ such that $d$ has the following decomposition:
\begin{equation*}
d(a) = c_0d_\K(a) + \tilde{d}(a)
\end{equation*}
where $\tilde{d}$ is an approximately inner derivation.
\end{theo}

\begin{proof}
Let $d_0$ be the $0$-th Fourier component of $d$. It is an invariant derivation, so by Proposition \ref{invariant_der_A} we have the unique decomposition:
\begin{equation*}
d_0(a)= c_0d_\K(a) + \tilde{d}_0(a)= c_0[\K, a] + \tilde{d}_0(a),
\end{equation*}
for every $a\in \mathcal{A}$, where $\tilde{d}_0$ is an approximately inner derivation.
From Proposition \ref{covariant_der_A} we have that the Fourier components $d_n$, $n \neq 0$ are inner derivations.   It follows from equation \eqref{Ces_eq}, by extracting $d_0$, that we have:
\begin{equation*}
d(a)=d_0(a)+\lim_{M\to\infty}\frac 1{M+1} \sum_{j=1}^M \left(\sum_{|n|\leq j,\, n\ne 0}d_n(a)\right).
\end{equation*}
The terms under the limit sign are all finite linear combinations of $n$-covariant derivations and so they are inner derivations themselves, meaning that the limit is approximately inner, which ends the proof.
\end{proof}

We also have the following useful but weaker convergence result for the Fourier components of derivations.
\begin{prop}\label{Fourier_comp_converge_A}
Let $d:\mathcal{A}\to A$ be any derivation.  Then for every $x\in c_{00}(\Z_{\ge0})\subseteq\ell^2(\Z_{\ge0})$ and $a\in\mathcal{A}$,
\begin{equation*}
\sum_{n\in\Z}d_n(a)x = d(a)x.
\end{equation*}
We say that $\sum_{n\in\Z}d_n(a)$ converges densely pointwise on the set $c_{00}(\Z_{\ge0})$.
\end{prop}

\begin{proof}
By Leibniz rule we only need to verify the above formula on generators of $\mathcal A$.  Moreover, it is enough to consider only $x=E_k^+$, since $c_{00}(\Z_{\ge0})$ consists of finite linear combinations of such $x$'s.  Below we show the details for $a=M_f^+$, as the calculations for  $a=U$ and $a=U^*$ are very similar.  We have the following basis decomposition:
\begin{equation*}
d(a)E_k^+ = \sum_{j=0}^\infty \langle E_j^+, d_n(a)E_k^+\rangle E_j^+.
\end{equation*}
Using the definition of the $n$-th Fourier components $d_n$ of $d$ and the fact that $d_n$ are $n$-covariant, a direct calculation gives:
\begin{equation*}
\langle E_j^+, d_n(a)E_k^+\rangle = \left\{
\begin{aligned}
&\langle E_j^+, d(a)E_k^+\rangle &&\textrm{ if }n+k=j\\
&0 &&\textrm{ otherwise.}
\end{aligned}\right.
\end{equation*}
It follows that
\begin{equation*}
\left\langle E_j^+, \sum_{n\in\Z}^\infty d_n(a)E_k^+\right\rangle = \langle E_j^+, d(a)E_k^+\rangle,
\end{equation*}
completing the proof.
\end{proof}

\subsection{Derivations in $B$}

Next we classify derivations in $B$ starting with the invariant derivations.  It turns out that there are new types of invariant derivations in $B$ that were not present in $A$. We describe these in the following lemma.

\begin{lem}\label{partial_der}
Let $\partial: \mathcal{F}\to C(G)$ be any derivation such that 
$$\partial f\circ\varphi = \partial(f\circ\varphi)$$ 
for all $f\in\mathcal{F}$, which we call a $\varphi$ invariant derivation in $C(G)$.  Then there exists a unique invariant derivation $\delta_{\partial}:\mathcal{B}\to B$ such that 
$$\delta_{\partial}(V)=0 \textrm{ and }\ \delta_{\partial}(M_f)= M_{\partial f}.$$
\end{lem}

\begin{proof}
Since $(V,M_f)$ is a defining representation for $C(G)\rtimes_\varphi\Z$, the only relation in the polynomial algebra $\mathcal{B}$ is 
\begin{equation*}
VM_f V^{-1} =M_{f\circ\varphi}.
\end{equation*}
Define the $\delta_\partial$ on the generators as above by $\delta_{\partial}(V)=0$ and $\delta_{\partial}(M_f)= M_{\partial f}$.   Using the Leibniz rule we try to extend this definition to all $\mathcal{B}$.  To verify that $\delta_\partial$ is a well-defined derivation from $\mathcal{B}\to B$, we thus need to check that it preserves the relation.  Applying $\delta_\partial$ to both sides of the relation yields $M_{\partial f\circ\varphi} = M_{\partial(f\circ\varphi)}$, completing the proof.
\end{proof}

As with algebra $A$ there is a simple example of an invariant derivation  which is given by
\begin{equation*}
\delta_{\mathbb{L}}(b) = [\mathbb{L},a]
\end{equation*}
where $b\in\mathcal{B}$.  This derivation is well defined because $\mathcal{B}$ is the space of polynomials in $V$, $V^{-1}$, and $M_f$, and we have $[\mathbb{L},V]=V$, $[\mathbb{L},V^{-1}]= -V^{-1}$ and $[\mathbb{L},M_f]=0$.

\begin{lem}\label{delta_f_zero}
For any $f\in C(G)$ such that
\begin{equation*}
\int_G f\,d\mu =0,
\end{equation*}
there exists a unique derivation $\delta_f:\mathcal{B}\to B$ such that 
\begin{equation*}
\delta_f(V) = VM_f,\quad \delta_f(V^{-1}) = -M_fV^{-1},\quad \delta_f(M_g) = 0,
\end{equation*}
where $M_g\in \mathcal{B}_0$.  Moreover $d_f$ is an approximately inner invariant derivation.
\end{lem}
The proof is identical to that of Lemma \ref{der_f_zero}.

\begin{prop}\label{invariant_der_B}
Let $\delta:\mathcal{B}\to B$ be any invariant derivation, then there exist $c_0\in\C$ and a $\varphi$ invariant derivation in $C(G)$, $\partial: \mathcal{F}\to C(G)$, such that $\delta$ is of the unique form
\begin{equation*}
\delta(b) = c_0\delta_{\mathbb{L}}(b)+ \delta_\partial(b) + \tilde{\delta}(b)
\end{equation*}
where $\delta_\partial:\mathcal{B}\to B$ is the derivation defined in Lemma \ref{partial_der} and $\tilde{\delta}$ is approximately inner.
\end{prop}

\begin{proof}
Since $\delta$ is invariant, there exists $f_0\in C(G)$ such that 
\begin{equation*}
\delta(V) = VM_{f_0}.
\end{equation*}
Moreover, there exists a linear map $\partial:\mathcal{F}\to C(G)$ such that 
\begin{equation*}
\delta(M_f)=M_{\partial f}.
\end{equation*}
Applying $\delta$ to the relation $M_{fg} = M_fM_g$ gives
\begin{equation*}
M_{\partial(fg)} = M_{\partial f}M_g + M_fM_{\partial g}.
\end{equation*}
Hence $\partial$ satisfies the Leibniz rule and thus is a derivation.  Applying $\delta$ to both sides of the relation $VM_fV^{-1}=M_{f\circ\varphi}$ yields:
\begin{equation*}
\partial f\circ\varphi = \partial(f\circ\varphi),
\end{equation*}
i.e. $\partial$ is $\varphi$ invariant.

Now write $f_0=c_0 + \tilde{f}$ with $c_0\in\C$ and 
\begin{equation*}
\int_G\tilde{f}\,d\mu=0.
\end{equation*}
It follows that 
\begin{equation*}
\delta(b) = c_0\delta_{\mathbb{L}}(b)+ \delta_\partial(b) + \delta_{\tilde{f}}(b)
\end{equation*}
where $\delta_\partial:\mathcal{B}\to B$ is the derivation defined in Lemma \ref{partial_der} and $\delta_{\tilde{f}}$ is defined in Lemma \ref{delta_f_zero}.  Arguing as in the proof of Proposition \ref{invariant_der_A} we obtain that $\delta_{\mathbb{L}}$ is not approximately inner. To complete the proof we notice that a non-zero derivation $\delta_\partial$ cannot be approximately inner since $\mathcal{F}$ is commutative and hence has no non-zero inner and approximately inner derivations. This proves the uniqueness of the decomposition and finishes the proof of the proposition.
\end{proof}

Because the proof of classifying all $n$-covariant derivations in $B$ is essentially the same as in the case of $A$, we only state the result.

\begin{prop}\label{covariant_der_B}
Let $\delta:\mathcal{B}\to B$ be an $n$-covariant derivation where $n\neq0$.  There exists an $\eta(\mathbb{L})\in B_0$ such that
\begin{equation*}
\delta(b)=[U^n\eta(\mathbb{L}),b]\textrm{ for }n\neq0
\end{equation*}
Moreover $\delta$ is an inner derivation.
\end{prop} 

Finally, putting Propositions \ref{invariant_der_B} and \ref{covariant_der_B} together along with the comment Ces\`aro mean convergence result for Fourier components of $d$ we have the following result.

\begin{theo}\label{der_decomp_B}
Let $\delta:\mathcal{B}\to B$ be any derivation.  Then there exists $c_0\in\C$ and a $\varphi$ invariant derivation  $\partial:\mathcal{F}\to C(G)$ such that $\delta$ has the following unique decomposition:
\begin{equation*}
\delta(b) = c_0\delta_{\mathbb{L}}(b) + \delta_\partial(b) +\tilde{\delta}(b)
\end{equation*}
where $\delta_\partial$ is the derivation defined in Lemma \ref{partial_der} and $\tilde{\delta}$ is an approximately inner derivation.
\end{theo}

We also state here a dense pointwise convergence result for Fourier components $\delta_n$ of a derivation $\delta:\mathcal B\to B$, which is similar to Proposition \ref{Fourier_comp_converge_A} and has completely analogous proof.
\begin{prop}\label{Fourier_comp_converge_B}
Let $\delta:\mathcal{B}\to B$ be any derivation.  Then for every $x\in c_{00}(\Z)\subseteq\ell^2(\Z)$ and $b\in\mathcal{B}$,
\begin{equation*}
\sum_{n\in\Z} \delta_n(b)x = \delta(b)x,
\end{equation*}
and we say that $\sum_{n\in\Z}\delta_n(b)$ converges densely pointwise on the set $c_{00}(\Z)$.
\end{prop}

\section{Lifting Derivations}

The first important observation is that any derivation in  algebra $A$ preserves compact operators.

\begin{prop}\label{der_preserve_compact}
If $d:\mathcal{A}\to  A$ is a derivation, then $d: \mathcal{A}\cap\mathcal{K}\to \mathcal K$.
\end{prop} 
 
\begin{proof}
It is enough to prove that $d(P_0)$ is compact, where $P_0$ is the orthogonal projection onto the one-dimensional subspace spanned by $E^+_0$, because $\mathcal A\cap\mathcal K$ is comprised of linear combinations of expressions of the form $U^lP_0(U^*)^j$ and compactness would follow immediately from the Leibniz property. To see that $d(P_0)$ is compact, apply $d$ to both sides of the relation $P_0=P_0^2$ to obtain:
\begin{equation*}
d(P_0)=d(P_0)P_0+P_0d(P_0)\in \mathcal{K},
\end{equation*}
which completes the proof.
\end{proof}
 
As a consequence of Proposition \ref{der_preserve_compact}, if $d:\mathcal{A}\to A$ is a derivation in $A$, then $[d]: \mathcal{B}\to B$ defined by
\begin{equation*} 
[d](a+\mathcal{K}) := [d(a)]
\end{equation*}
gives a derivation in $B$, which, by Corollary \ref{structure_cor}, is defined on $\mathcal{B}\cong[\mathcal{A}]$.  

As a consequence to Proposition \ref{der_preserve_compact}, we have: 
$$[d_\K]=\delta_{\mathbb{L}}.$$   
Clearly, if $d$ is an approximately inner derivation, then so is $[d]$.  In general, given a derivation $\delta:\mathcal{B}\to B$, if there exists a derivation $d:\mathcal{A}\to A$ such that $[d]=\delta$ we call such a $d$ a {\it lift} of $\delta$.

A natural question is: which derivations $\delta:\mathcal{B}\to B$ can be lifted to a derivation $d:\mathcal{A}\to A$?  It follows from Theorems \ref{der_decomp_A} and \ref{der_decomp_B} that if there is a nonzero $\varphi$ invariant derivation in $C(G)$, $\partial:\mathcal{F}\to C(G)$, then there is no $d:\mathcal{A}\to A$ such that $[d]=\delta_\partial$, because $\delta_\partial$ is not approximately inner.  A natural example of this is $G=\T^1=\R/\Z$ with $x_k = \theta k\ (\textrm{mod }\Z)$, $k\in\Z$ and $\theta$ irrational, giving a dense subgroup of $\T^1$.  In this case, $\mathcal{F}$ is the actual space of trigonometric polynomials.  Any derivation $\partial:\mathcal{F}\to C(\T^1)$ invariant with respect to $\varphi(x) = x + \theta\ (\textrm{mod }\Z)$ is of the form:
\begin{equation*}
\partial(f) = \textrm{const }\frac{d}{dx}f(x).
\end{equation*}
In this case, the algebra $B$ is generated by $V$ and $W=M_{e^{2\pi ix}}$ satisfying the relation
\begin{equation*}
VW = e^{2\pi i\theta}WV.
\end{equation*}
Consequently, $B$ is isomorphic with the irrational rotation algebra.  $\mathcal{B}$ is the algebra of polynomials in $V$ and $W$ and the derivation $\delta_{d/dx}:\mathcal{B}\to B$ is given on generators by
\begin{equation*}
\delta_{d/dx}(V) = 0\textrm{ and }\delta_{d/dx}(W) = 2\pi i W
\end{equation*}
and it cannot be lifted to a derivation in $A$.  The key reason is that there is an additional relation on $A$ given by equation (\ref{M_with_P_zero}) which prevents existence of such a lift.  We conjecture however, that for any compact infinite monothetic group, any approximately inner derivation $\delta:\mathcal{B}\to B$ can be lifted to a derivation $d:\mathcal{A}\to A$.

For the remainder of the section we let $G$ be totally disconnected, in other words $G$ is an odometer, and thus by Proposition \ref{tot_disc_mono}, there exists an infinite supernatural number $N$ such that $G\cong \Z/N\Z$.  It was proved in \cite{KMRSW2} that for such $G$'s, the algebras $A$ and $B$ are precisely the Bunce-Deddens-Toeplitz, $A(N)$, and Bunce-Deddens algebras, $B(N)$, respectively.  It follows from Theorem 4.4 in \cite{KMRSW2} that there are no nontrivial $\varphi$ invariant derivations $\partial:\mathcal{F}\to C(G)$.  Below we prove one of the main results of this paper that for odometers, any unbounded derivation in $B(N)$ can be lifted to an unbounded derivation in $A(N)$.

We will need the following useful result for computing Hilbert-Schmidt norms of operators in $\ell^2(\Z)$ and $\ell^2(\Z_{\ge0})$.  Since below we work mostly with algebra $A$, we only state the corresponding version for brevity.

\begin{prop}\label{HSProp}
Let $a:\ell^2(\Z_{\ge0})\to\ell^2(\Z_{\ge0})$ be defined by:
\begin{equation*}
a = \sum_{n=0}^\infty U^na_n(\K) + \sum_{n=-\infty}^{-1}a_n(\K)(U^*)^{-n},
\end{equation*}
where $\{a_n(k)\}_{n\in\Z, k\in \Z_{\ge0}}\in\ell^2(\Z\times\Z_{\ge0})$.  Then $a$ is an integral operator with the Hilbert-Schmidt norm given by:
\begin{equation*}
\|a\|_{\textrm{HS}}^2 = \sum_{n\in\Z}\sum_{k=0}^\infty |a_n(k)|^2.
\end{equation*}
\end{prop}

\begin{proof}
Write $f\in\ell^2(\Z_{\ge0})$ in the canonical basis:
\begin{equation*}
f = \sum_{k=0}^\infty f(k)E_k^+.
\end{equation*}
Applying the formula for $a$ to  $f$ yields:
\begin{equation*}
af = \sum_{n\ge0}\sum_{k\ge0}a_n(k)f(k)E_{k+n}^+ + \sum_{0\le k+n}\sum_{n<0}a_n(k+n)f(k)E_{k+n}^+.
\end{equation*}
Resumming both terms gives:
\begin{equation*}
\begin{aligned}
af &= \sum_{n\ge0}\sum_{n\ge k}a_{n-k}(k)f(k)E_n^+ + \sum_{n\ge0}\sum_{n<k}a_{n-k}(n)f(k)E_n^+\\
&= \sum_{n\ge0}\left(\sum_{k=0}^\infty a_{n-k}(\textrm{min}\{n,k\})f(k)\right)E_n^+.
\end{aligned}
\end{equation*}
This shows that $a$ is an integral operator with integral kernel 
\begin{equation*}
\kappa(k,n) = a_{n-k}(\textrm{min}\{n,k\})\in\ell^2(\Z\times\Z_{\ge0}).
\end{equation*}
Therefore, by writing $a$ in the following way
\begin{equation*}
(af)(n) = \sum_{k=0}^\infty\kappa(k,n)f(k) = \sum_{k=0}^na_{n-k}(k)f(k) + \sum_{k>n}a_{n-k}(n)f(k),
\end{equation*}
the Hilbert-Schmidt norm formula now follows, completing the proof.
\end{proof}

\begin{theo}\label{lift_theo}
Let $\delta:\mathcal{B}(N)\to B(N)$ be any derivation.  There exists a derivation $d:\mathcal{A}(N)\to A(N)$ such that $[d]=\delta$.
\end{theo}
\begin{proof}
Let $\delta:\mathcal{B}(N)\to B(N)$ be an approximately inner derivation in $B(N)$, then by Theorem \ref{der_decomp_B} and by Propositions \ref{invariant_der_B} and \ref{covariant_der_B} we have
\begin{equation*}
\delta(b) = \left[\sum_{n\ge0}V^nM_{f_n} + \sum_{n<0}M_{f_n}V^n,b\right]
\end{equation*}
for $b\in\mathcal{B}(N)$, where the convergence of infinite sums is understood as being densely pointwise on $c_{00}(\Z)$.  In order to construct a lift of $\delta$ we need to consider derivations $d:\mathcal{A}(N)\to A(N)$ given by the following expression, densely pointwise convergent on $c_{00}(\Z_{\geq 0})$:
\begin{equation*}
d(a) = \left[\sum_{n\ge0}U^n\beta_n(\K) + \sum_{n<0}\beta_n(\K)(U^*)^{-n},a\right]
\end{equation*}
for $a\in\mathcal{A}(N)$, where $\beta_n(k)$ has the following decomposition: 
$$\beta_n(\K) = \beta_{n,0}(\K) + M^+_{f_n}$$ 
where $\beta_{n,0}(k)\in c_0(\Z_{\geq 0})$.   We need to find conditions on $\beta_{n,0}(k)$ so that $d$ is a well-defined derivation in $A(N)$ such that 
$$[d](a)=\delta([a])$$  
for all $a\in\mathcal A(N)$.
By the Leibniz rule we only need to check this equation on the generators $U$, $U^*$, and $M^+_\chi$, where $\chi$ is a character on $\Z/N\Z$.

A direct computation yields the following formula:
\begin{equation*}
\begin{aligned}
&d(U) - T(\delta(V)) = \sum_{n\ge0}U^{n+1}\left(\beta_{n,0}(\K+I)-\beta_{n,0}(\K)\right) \\
&\quad+ \sum_{n<0}\left(\beta_{n,0}(\K)-\beta_{n,0}(\K-I)\right)(U^*)^{-n-1} + P_0\left(\sum_{n<0}M^+_{f_{n}\circ\varphi^{-1}}(U^*)^{-n-1}\right)
\end{aligned}
\end{equation*}
Similarly, on $U^*$ we have:
\begin{equation*}
\begin{aligned}
&d(U^*) - T(\delta(V^{-1})) = \sum_{n>0}U^{n-1}\left(\beta_{n,0}(\K-I)-\beta_{n,0}(\K)\right) \\
&\quad+ \sum_{n\le0}\left(\beta_{n,0}(\K)-\beta_{n,0}(\K+1)\right)(U^*)^{-n+1} + \left(\sum_{n>0}U^{n-1}M^+_{f_{n}\circ\varphi^{-1}}\right)P_0
\end{aligned}
\end{equation*}
Finally, we get the following expression for diagonal operators $M^+_\chi$:
\begin{equation*}
\begin{aligned}
&d(M^+_\chi) - T(\delta(M_\chi)) = \sum_{n\ge0}U^n\beta_{n,0}(\K)(M^+_\chi-M^+_{\chi\circ\varphi^n}) \\
&\quad+ \sum_{n<0}\beta_{n,0}(\K)(M^+_{\chi\circ\varphi^n}-M^+_\chi)(U^*)^{-n}.
\end{aligned}
\end{equation*}
The result follows provided we can choose $\beta_{n,0}(k)$ so that the right-hand sides of the above equations are compact operators.  We compute the Hilbert-Schmidt norm of the above operators to show the compactness.  A direct calculation using Proposition \ref{HSProp} yields the following formulas:
\begin{equation*}
\begin{aligned}
I&:=\|d(M^+_\chi) - T(\delta(M_\chi))\|^2_{\textrm{HS}} = \sum_{n\in\Z}\sum_{k\ge0}|\beta_{n,0}(k)|^2|\chi(x_k)-\chi(x_{k+n})|^2 \\
II&:=\|d(U) - T(\delta(V))\|^2_{\textrm{HS}} = \sum_{n\in\Z}\sum_{k\ge1}|\beta_{n,0}(k)-\beta_{n,0}(k-1)|^2 + \sum_{n\in\Z}|\beta_{n,0}(0)-f_n(x_{-1})|^2\\
&\qquad\qquad\qquad\qquad\qquad\,\,\,=\|d(U^*) - T(\delta(V^{-1}))\|^2_{\textrm{HS}}.
\end{aligned}
\end{equation*}
 We define $\beta_{n,0}(k)$ to have the following form:
\begin{equation*}
\beta_{n,0}(k)=\left\{
\begin{aligned}
&-f_n(x_{-1})\left(\frac{N_n-k}{N_n}\right) &&0\le k<N_n\\
&0 &&k\ge N_n,
\end{aligned}\right.
\end{equation*}
where the numbers $N_n$ will be chosen later.

Notice that any character on $\Z/N\Z$ is of the form:
$$\chi(x_k) = \textrm{exp}\left(\frac{2\pi ijk}{M}\right),$$ 
where $M\mid N$ and $j\in\Z$.  Therefore $I$ and $II$ become
\begin{equation*}
\begin{aligned}
I&=\sum_{n\in\Z}\sum_{k=0}^{N_n}|f_n(x_{-1})|^2\left(\frac{N_n-k}{N_n}\right)^2|1-e^{2\pi ijn/M}|^2\\
II&=\sum_{n\in\Z}\sum_{k=0}^{N_n-1}\frac{|f_n(x_{-1})|^2}{N_n^2} = \sum_{n\in\Z}\frac{|f_n(x_{-1})|^2}{N_n}.
\end{aligned}
\end{equation*}

The key observation used below is that the coefficients $f_n$ on the Fourier decomposition of the derivation $\delta:\mathcal B\to B$ satisfly the following condition:
for all $M\mid N$:
\begin{equation}\label{hs_condition}
\sum_{M\nmid n}|f_n(x_{-1})|^2<\infty.
\end{equation}
This follows from the formula:
\begin{equation*}
\sum_{n\ge0}|f_n(x_{-1})|^2|1-e^{2\pi ijn/M}|^2= \|P_{\geq0}\delta(\textrm{exp}(2\pi ij\mathbb{L}/M)P_{-1}\|^2_{\textrm{HS}}<\infty,
\end{equation*}
and a similar formula for $n<0$:
\begin{equation*}
\sum_{n\in\Z}|f_n(x_{-1})|^2|1-e^{2\pi ijn/M}|^2=\|P_{-1}\delta(\textrm{exp}(2\pi ij\mathbb{L}/M)P_{\geq0}\|^2_{\textrm{HS}}<\infty.
\end{equation*}
Here $P_{-1}$ is the orthogonal projection in $\ell^2(\Z)$ onto the one-dimensional subspace spanned by $E_{-1}$, while $P_{\geq0}$ is the orthogonal projection onto the subspace spanned by $\{E_{l}\}_{l\geq0}$.
Equations above imply that we have:
\begin{equation*}
\begin{aligned}
\infty&>\sum_{n\in\Z}|f_n(x_{-1})|^2|1-e^{2\pi ijn/M}|^2 = \sum_{M\nmid n}|f_n(x_{-1})|^2|1-e^{2\pi ijn/M}|^2\\
&>\textrm{const}\sum_{M\nmid n}|f_n(x_{-1})|^2,
\end{aligned}
\end{equation*}
since the factor $1-e^{2\pi ijn/M}$ has only finitely many values.  This gives the following estimate:
\begin{equation*}
\begin{aligned}
I&=\sum_{n\in\Z}\sum_{k=0}^{N_n}|f_n(x_{-1})|^2\left(\frac{N_n-k}{N_n}\right)^2|1-e^{2\pi ijn/M}|^2\\
&\le4\sum_{M\nmid n}\sum_{k=0}^{N_n}|f_n(x_{-1})|^2\left(\frac{N_n-k}{N_n}\right)^2\sim\textrm{const} \sum_{M\nmid n}N_n|f_n(x_{-1})|^2.
\end{aligned}
\end{equation*}

%For the supernatural number $N$ we write $\mathcal J_N = \{\textrm{finite divisors of }N\}$. Below, we write any $n\in\Z$ as the product:
%$$n=jn',$$
%where $j=\textrm{gcd}(n,N)$.  Using this decomposition we choose $N_n =C_s$ to be a constant depending on $s$ only. Consequently, we have the following expressions:
%\begin{equation*}
%\begin{aligned}
%I&\le 4\sum_{s\in S_N,M\nmid s}\sum_{\textrm{gcd}(n',N)=1}C_s|f_n(x_{-1})|^2\\
%II&=\sum_{s\in S_N}\sum_{\textrm{gcd}(n',N)=1}\frac{|f_n(x_{-1})|^2}{C_s}.
%\end{aligned}
%\end{equation*}
%Notice that, by equation (\ref{hs_condition}), we have:
%\begin{equation}\label{hs_condition_two}
%\sum_{s\in S_N}\sum_{\textrm{gcd}(n',N)=1}|f_n(x_{-1})|^2 <\infty.
%\end{equation}

To proceed further we choose a scale  $s = (s_m)_{m\in\N}$ for the supernatural number $N$, which  is a sequence of positive integers such that $s_m$ divides $s_{m+1}$, $s_m<s_{m+1}$, and such that $N=\lim_m s_m$, see \eqref{N_limit}. For every $n\in\Z$ there is an index $m$ such that $s_m\mid n$ but $s_{m+1}\nmid n$. We then write  
$$n=s_mn',$$
where $n'$ is such that $s_{m+1}/s_m\nmid n'$.  Using this decomposition we choose $N_n =C_m$ to be a constant depending on $m$ only, to be determined later. Also, without loss of generality, we can choose $M$, in the formula for the character $\chi$, to be equal to one of the elements of the scale: $M=s_q$. It is then important to notice that $s_q\nmid n=s_mn'$ if and only if $m<q$.
Consequently, we have the following expressions:
\begin{equation*}
\begin{aligned}
I&\le\textrm{const}\sum_{m=1}^{q-1} C_m\sum_{s_{m+1}/s_m\nmid n'}|f_{s_mn'}(x_{-1})|^2 =\textrm{const}\sum_{m=1}^{q-1} C_m\sum_{s_m\mid n, \,s_{m+1}\nmid n}|f_{n}(x_{-1})|^2 \\
&\le \textrm{const}\sum_{m=1}^{q-1} C_m\sum_{s_{m+1}\nmid n}|f_{n}(x_{-1})|^2 <\infty
\end{aligned}
\end{equation*}
for any choice of $C_m$ because the sum over $s_{m+1}\nmid n$ is finite by equation (\ref{hs_condition}).

Next, for $II$ we have an estimate:
\begin{equation*}
II = \sum_{m=1}^{\infty} \frac{1}{C_m}\sum_{s_m\mid n, \,s_{m+1}\nmid n}|f_{n}(x_{-1})|^2\le \sum_{m=1}^{\infty} \frac{1}{C_m}\sum_{s_{m+1}\nmid n}|f_{n}(x_{-1})|^2.
\end{equation*}
By equation (\ref{hs_condition}) the interior sum is finite. Finally, we can always choose $C_m$ large enough so that $II<\infty$.  This completes the proof.
\end{proof}


\begin{thebibliography}{99}


%\bibitem{APS}
%Atiyah, M. F., Patodi, V. K. and Singer I. M., Spectral asymmetry and Riemannian geometry I, II, III, {\it Math. Proc. Camb. Phil. Soc.} 77(1975) 43-69, 78(1975) 43-432, 79(1976) 71-99.

%\bibitem{BH}
%Barría, J. and P. R. Halmos, P.R., Asymptotic Toeplitz operators, {\it Trans. Amer. Math. Soc.} 273 , 621 - 630, 1982 

%\bibitem{BS}
%Battisti, U., Seiler, J., Boundary value problems with Atiyah-Patodi-Singer type conditions and spectral triples,
%arXiv:1503.02897 [math.AP].

%\bibitem{BBW}
%Boo-Bavnbek, B. and Wojciechowski, K.P., {\it Elliptic Boundary Problems for Dirac Operators}, Birkhauser, Boston, 1993.

\bibitem{B}
Bratteli, O., {\it Derivations, Dissipations and Group Actions on C$^*$ -algebras}, Lecture Notes in Math. 1229, Springer, 1986.

\bibitem{BEJ}
Bratteli, O., Elliott, G. A., and Jorgensen, P. E. T., Decomposition of unbounded derivations into invariant and approximately inner parts,
{\it Jour. Reine Ang. Math.}, 346, 166 - 193, 1984.
%
%\bibitem{BHS}
%Brown, A., Halmos, P.R., Shields A.L., Ces\`{a}ro operators,
%{\it Acta Sci. Math. (Szeged)}, 26, 125 - 137, 1965.

%\bibitem{CKW}
%Carey, A. L., Klimek, S. and Wojciechowski, K. P., Dirac operators on noncommutative manifolds with boundary, 
%{\it Letts. Math. Phys.}, 93, 107 - 125, 2010.

%\bibitem{Cob}
%Coburn, L. A. The C$^*$-algebra generated by an isometry, 
%{\it Bull. Amer. Math. Soc.} 73, 722 - 726, 1967

%\bibitem{Connes}
%Connes, A., {\it Non-Commutative Differential Geometry}, Academic Press, 1994.

\bibitem{D} 
Downarowicz, T., Survey of odometers and Toeplitz flows, {\it Contemp. Math.}, 385, 7 - 37, 2005. 

%\bibitem{FGMR}
%Forsyth, I., Goffeng, M., Mesland, B., Rennie, A., Boundaries, spectral triples and K-homology,
%arXiv:1607.07143 [math.KT].

\bibitem{F}
Fillmore, P. {\it A User's Guide to Operator Algebras}, Wiley-Interscience Publication, 1996.

%\bibitem{GH}
%Gottschalk, W. and Hedlund, G., Topological dynamics, AMS Colloquium Publications, AMS, 36, 1955. 
%
%\bibitem{H}
%Hadfield, T., The noncommutative geometry of the discrete Heisenberg group,
%{\it Houston J. Math.}, 29, 453 - 481, 2002.

\bibitem{HS}
Halmos, P. and Samelson, H., On Monothetic Groups, {\it Proc. AMS}, 28, 254 - 258, 1942.

%\bibitem{HR}
%Hewitt, E, and Ross, K., {\it Abstract Harmonic Analysis}, Springer-Verlag, 1979.

%\bibitem{HS}
%Halmos, P.R. and Sunder V.S., {\it Bounded Integral Operators on $L^2$ Spaces}, Springer-Verlag, 1978.

%\bibitem{IL}
%Iochum, B., Levy C., Spectral triples and manifolds with boundary. {\it J. Funct. Anal.} 260, 117?134,  2011

%\bibitem{JR}
%Jaffe, A. and Ritter, G. Reflection Positivity and Monotonicity, {\it Jour. Math. Phys.}, 49, 052301, 2008.
%
%\bibitem{J}
%Jorgensen, P., Approximately inner derivations, decompositions and vector fields of simple C$^*$-algebras, in {\it Mappings of operator algebras: Proceedings of the Japan-U.S. Joint Seminar (Philadelphia, 1988)}, pp. 15 - 113, (H. Araki and R.V. Kadison, eds.), Progr. Math., vol. 84, Birkhauser, Boston, 1990.

\bibitem{KH}
Katok, A. and Hasselblatt, B., {\it Introduction to the modern theory of dynamical systems}, Cambridge University Press, 1996.

%\bibitem{KR}
%Kadison, R. V.; Ringrose, J. R., {\it Fundamentals of the Theory of Operator Algebras}, Academic Press 1986.

%\bibitem{KL}
%Klimek, S. and Lesniewski, A., Quantum Riemann surfaces, I. The unit disk,
%{\it Comm. Math. Phys.}, 146, 103 - 122, 1992.

%\bibitem{KL1}
%Klimek, S. and Lesniewski, A., Quantum Riemann surfaces, II. The discrete series, {\it Lett. Math. Phys.}, 24, 125--139, 1992.
%
%\bibitem{KL2}
%Klimek, S. and Lesniewski, A., A Two-Parameter Quantum Deformation of the Unit Disc, {\it Jour. Func. Anal.}, 115, 1 - 23, 1993.
%
%\bibitem{KL3}
%Klimek, S. and Lesniewski, A., Quantum Riemann surfaces, III. The Exceptional Cases, {\it Lett. Mat. Phys.}, 32, 45 - 61, 1994.

%\bibitem{KL4}
%Klimek, S. and Lesniewski, A., Quantum Riemann Surfaces for Arbitrary Planck's Constant, {\it J. Math. Phys.}, 37, 2157 - 2165, 1996.
%
%\bibitem{KM1}
%Klimek, S. and McBride, M., D-bar Operators on Quantum Domains, 
%{\it Math. Phys. Anal. Geom.}, 13, 357 - 390, 2010.

%\bibitem{KMR}
%Klimek, S., McBride, M., and Rathnayake, S., Derivations and Spectral Triples on Quantum Domains II: Quantum Annulus, to appear {\it Sci. Chi. Math.}, arXiv:1710.06257.

\bibitem{KMRS}
Klimek, S., McBride, M., Rathnayake, S., and Sakai, K., The Quantum Pair of Pants,
{\it SIGMA}, 11, 012, 1 - 22, 2015. 

\bibitem{KMRSW1}
Klimek, S., McBride, M., Rathnayake, S., Sakai, K., and Wang, H., Derivations and Spectral Triples on Quantum Domains I: Quantum Disk, 
{\it SIGMA}, 13, 075, 1 - 26, 2017.

\bibitem{KMRSW2}
Klimek, S., McBride, M., Rathnayake, S., Sakai, K., Wang, H., Unbounded Derivations in Bunce-Deddens-Toeplitz Algebras, {\it Jour. Math. Anal. Appl.}, 15, 988 - 1020, 2019.

%\bibitem{KM2}
%Klimek, S. and McBride, M., A note on Dirac Operators on the Quantum Punctured Disk. {\it SIGMA}, 6, 056, 2010.
%
%\bibitem{KM3}
%Klimek, S. and McBride, M., Classical limit of the d-bar operators on quantum domains, {\it J. Math. Phys. } 52, 093501, 2011.
%
%\bibitem{KM4}
%Klimek, S. and McBride, M., A Note on Gluing Dirac Type Operators on a Mirror Quantum Two-Sphere, {\it SIGMA}, 10, 036, 2014.
%
%\bibitem{KM5}
%Klimek, S. and McBride, M., Global boundary conditions for a Dirac operator on the solid torus. {\it Jour. Math. Phys.}, 52, 1 - 14, 2011.

\bibitem{K}
Kurka, P., {\it Topological and Symbolic Dynamics}, Soci{\'e}t{\'e} math{\'e}matique de France, 2003.

\bibitem{M}
Morris, S., {\it Pontryagin Duality and the Structure of Locally Compact Abelian Groups}, Cambridge University Press, 1977

\bibitem{P}
Pedersen, G. K., Lifting Derivations from Quotients of Separable C$^*$-algebras, {\it Proc. Natl. Acad. Sci.}, 73, 1414 - 1415, 1976

%\bibitem{R}
%Renault J., Cartan subalgebras in C*-algebras, 
%{\it Irish Math. Soc. Bull.}, 61, 29 - 63, 2008.

\bibitem{Robert}
Robert, A., {\it A Course in $p$-adic Analysis}, Springer, 2000.

%\bibitem{RS}
%Ribes, L. and Zalesskii, P., {\it Profinite Groups}, Springer, 2000.

\bibitem{S}
Sakai, S., {\it Operator Algebras in Dynamical Systems}, Cambridge University Press, 1991.

%\bibitem{V} 
%de Vries, J., Elements of topological dynamics, Math. Appl., Kluwer Academic Publishers Group, 257, 1993.

\bibitem{W}
Willard, S., {\it General Topology}, Addison-Wesley Publishing, 1970

\end{thebibliography}
\end{document}